\documentclass[amsthm]{monsky2009} 

\usepackage{graphicx}

\usepackage{amssymb}
\usepackage{amsmath}

\usepackage{bm}

\newtheorem*{corollary*}{Corollary}
\newtheorem*{remark*}{Remark}

\newtheorem{definition}{Definition}
\newtheorem{theorem}[definition]{Theorem}
\newtheorem{lemma}[definition]{Lemma}

\newtheorem{corollary}[definition]{Corollary}

\newcommand{\id}{\ensuremath{\mathrm{id}}}

\newcommand{\newo}{\ensuremath{\mathcal{O}}}

\newcommand{\newq}{\ensuremath{\mathbb{Q}}}

\newcommand{\newz}{\ensuremath{\mathbb{Z}}}

\begin{document}

\begin{frontmatter}






\title{A result of Lemmermeyer on class numbers}
\author{Paul Monsky}

\address{Brandeis University, Waltham MA  02454-9110, USA\\  monsky@brandeis.edu }

\begin{abstract}
I present Franz Lemmermeyer's proof that if $p$ is a prime $\equiv 9 \pod{16}$ then the class number of $\newq\left(p^{\frac{1}{4}}\right)$ is $\equiv 2 \pod{4}$.

\end{abstract}
\maketitle

\end{frontmatter}


Let $p$ be a prime $\equiv 1 \pod{4}$. Then the class number of $k=\newq\left(\sqrt{p}\right)$ is odd, and the fundamental unit of $\newo_{k}$ has norm $-1$; this result in essence goes back to Gauss. Years ago I conjectured that if $p\equiv 9\pod{16}$ then the class number of $\newq\left(p^{\frac{1}{4}}\right)$ is $\equiv 2\pod{4}$. (Parry \cite{2} had previously shown that it's even, and that when 2 is not a fourth power in $\newz/p$ it's $\equiv 2\pod{4}$.) I gave a proof of my conjecture assuming that the elliptic curve $y^{2}=x^{3}-px$ has positive rank, as the Birch Swinnerton-Dyer conjecture predicts.

Recently I asked on Mathoverflow whether the elliptic curve assumption could be eliminated. Franz Lemmermeyer responded with an unconditional proof that starts with Gauss' result and continues with two applications of the ambiguous class number formula. His very nice argument deserves wider circulation, so I'm writing it up here.

\begin{theorem}
\label{theorem1}
If $p\equiv 1\pod{8}$, $\newq\left(p^{\frac{1}{4}}\right)$ has even class number.
\end{theorem}

\begin{proof}[Proof (Lemmermeyer)]
Let $F$ be the quartic subfield of $\newq(\mu_{p})$. Then $F\supset k=\newq(\sqrt{p})$.  Since $p\equiv 1\pod{8}$, the infinite prime of $\newq$ is unramified in $F$, and the only prime of $\newq$ that ramifies in $F$ is $(p)$.

Since $F\left(p^{\frac{1}{4}}\right)$ is the compositum of $\newq\left(p^{\frac{1}{4}}\right)$ and $F$ it is a Galois extension of $k$ with Galois group $\newz/2\times\newz/2$. Since $p\ne 2$, the ramification is tame, and the prime above $p$ cannot ramify totally in the extension. It follows that $\left(p^{\frac{1}{4}}\right)$ cannot ramify from $\newq\left(p^{\frac{1}{4}}\right)$ to $F\left(p^{\frac{1}{4}}\right)$. So $\newq\left(p^{\frac{1}{4}}\right)$ has an everywhere unramified extension, $F\left(p^{\frac{1}{4}}\right)$, of degree 2, and class-field theory gives the result.
\end{proof}

\begin{corollary}
\label{corollary2}
Suppose $p\equiv 1\pod{8}$ and $F$ is as in Theorem \ref{theorem1}. If $F\left(p^{\frac{1}{4}}\right)$ has odd class number then the class number of $\newq\left(p^{\frac{1}{4}}\right)$ is $\equiv 2\pod{4}$.
\end{corollary}

\begin{proof}
Suppose on the contrary that 4 divides the class number. Then $\newq\left(p^{\frac{1}{4}}\right)$ admits an unramified abelian extension of degree 4. Translating by $F$ we get a degree 2 unramified extension of $F\left(p^{\frac{1}{4}}\right)$, contradicting the odd class number assumption.
\end{proof}

\begin{lemma}
\label{lemma3}
The $F$ of Theorem \ref{theorem1} is the unique degree 2 extension of $k$ unramified outside of $(\sqrt{p})$.
\end{lemma}

\begin{proof}
Let $F^{\prime}$ be a second such extension. Since the class number of $k$ is odd, $k$ has no unramified extensions of degree 2 and $(\sqrt{p})$ must ramify in $F^{\prime}$. There is a third quadratic extension, $F^{\prime\prime}$, of $k$ contained in $FF^{\prime}$ and the same argument shows that $(\sqrt{p})$ ramifies in $F^{\prime\prime}$. So $(\sqrt{p})$ ramifies totally in $FF^{\prime}$, contradicting tameness.
\end{proof}

\begin{theorem}
\label{theorem4}
If $p\equiv 1\pod{8}$ and $\epsilon > 0$ is a unit of norm $-1$ in the ring of integers of $k$, then $F=k\left(\sqrt{\epsilon\sqrt{p}}\right)$.
\end{theorem}

\begin{proof}
$\epsilon\sqrt{p}$ and its $\newq$-conjugate $\epsilon^{-1}\sqrt{p}$ are both $>0$. So neither of the infinite primes of $k$ ramify in $k\left(\sqrt{\epsilon\sqrt{p}}\right)$. If $r^{2}-sp^{2}=-4$, $r$ and $s$ cannot both be odd. It follows that $\epsilon = a+b\sqrt{p}$ with $a$ and $b$ integers. Also, $0<\epsilon + \epsilon^{-1}=2b\sqrt{p}$, and $b>0$. Now $a^{2}-pb^{2}=-1$. Since $pb^{2}\equiv 1\pod{8}$, 8 divides $a^{2}$ and 4 divides $a$. Furthermore every prime that divides $b$ divides $a^{2}+1$, and so is $\equiv 1\pod{4}$. Since $b>0$, $b\equiv 1\pod{4}$. Let $P$ be a prime of $\newo_{k}$ lying over $(2)$. Then in the $P$-completion of $\newo_{k}$, $\epsilon\sqrt{p}=bp+a\sqrt{p}\equiv 1\pod{4}$. So $P$ does not ramify in $k\left(\sqrt{\epsilon\sqrt{p}}\right)$. It follows that the only prime that can ramify in $k\left(\sqrt{\epsilon\sqrt{p}}\right)$ is $(\sqrt{p})$, and we apply Lemma \ref{lemma3}.
\end{proof}

\begin{corollary}
\label{corollary5}
$F\left(p^{\frac{1}{4}}\right)=F(\sqrt{\epsilon})$.
\end{corollary}

\begin{proof}
Both fields are degree 2 extensions of $F$. Since $\sqrt{\epsilon\sqrt{p}}$ is in $F$, $\sqrt{\epsilon}$ is in $F\left(p^{\frac{1}{4}}\right)$.
\end{proof}

We now give the idea of Lemmermeyer's proof. The class number of $k$ is known to be odd. Lemmermeyer uses the ambiguous class number formula to deduce that $k(\sqrt{\epsilon})$ has odd class number. Then assuming $p\equiv 9\pod{16}$ he uses it once more to show that $F(\sqrt{\epsilon})$ has odd class number. Corollaries \ref{corollary5} and \ref{corollary2} complete the proof.

We introduce some notation. Suppose $L\supset K$ is a degree 2 extension of number fields with Galois group $G =\{\id, \sigma\}$. $U_{K}$ consists of the units of $\newo_{K}$ while $h_{L}$ and $h_{K}$ are the class numbers of $L$ and $K$. $C_{L}$ is the class group of $L$, while $C_{L}^{G}$, the ambiguous class group, consists of the elements of $C_{L}$ fixed by $\sigma$. The  following result is contained in Theorem 4.1 of \cite{1}.

\begin{theorem}
\label{theorem6}
$|C_{L}^{G}|=h_{K}\cdot(2^{t-1}/j)$ where $t$ is the number of primes of $K$, finite or infinite, that ramify in $L$, while $j$ is the index in $U_{K}$ of the subgroup consisting of elements that are norms from $L$. (Since this subgroup contains $U_{K}^{2}$, $j$ is a power of 2.)  Furthermore, if $|C_{L}^{G}|$ is odd, $h_{L}$ is odd.
\end{theorem}

\begin{lemma}
\label{lemma7}
Just two primes of $k$ ramify in $k(\sqrt{\epsilon})$.
\end{lemma}

\begin{proof}
$\epsilon > 0$, and the $\newq$-conjugate $-\epsilon^{-1}$ of $\epsilon$ is $<0$. So one of the two infinite primes ramifies. Also $\epsilon = a+b\sqrt{p}$ with $a\equiv 0\pod{4}$, $b\equiv 1\pod{4}$. So $\epsilon\equiv 1\pod{4}$ in the completion of $\newo_{k}$ at $\left(2,\frac{1-\sqrt{p}}{2}\right)$, and $\epsilon\equiv -1\pod{4}$ in the completion of $\newo_{k}$ at $\left(2,\frac{1+\sqrt{p}}{2}\right)$.  Finally no other primes can ramify.
\end{proof}

\begin{theorem}
\label{theorem8}
$k(\sqrt{\epsilon})$ has odd class number.
\end{theorem}

\begin{proof}
Theorem \ref{theorem6} and Lemma \ref{lemma7} show that the ambiguous class number for the extension $k(\sqrt{\epsilon})\supset k$ is $\frac{2h_{k}}{j}$. So it suffices to show that $j>1$. Now $-1$ is in $U_{k}$. But as we saw above there is an infinite prime of $k$ ramifying in $k(\sqrt{\epsilon})$, and $-1$ evidently is not a local norm at that prime.
\end{proof}

\begin{lemma}
\label{lemma9}
$\epsilon$ represents a primitive fourth root of unity in $\newo_{k}/(\sqrt{p})=\newz/p$. Furthermore the prime $(\sqrt{p})$ of $k$ splits in $k(\sqrt{\epsilon})$.
\end{lemma}

\begin{proof}
$\epsilon=a+b\sqrt{p}$ with $a^{2}-pb^{2}=-1$. So mod $\sqrt{p}$, $\epsilon^{2}\equiv a^{2}\equiv -1$, giving the first result. Since $p\equiv 1\pod{8}$, any fourth root of unity in $(\newz/p)^{*}$ is a square, and the second result follows.
\end{proof}

\begin{theorem}[Lemmermeyer]
\label{theorem10}
Suppose $p\equiv 9\pod{16}$. Then the ambiguous class number for the extension $F(\sqrt{\epsilon})\supset k(\sqrt{\epsilon})$ is odd. So $F(\sqrt{\epsilon})$ has odd class number, and Corollaries \ref{corollary5} and \ref{corollary2} show that the class number of $\newq\left(p^{\frac{1}{4}}\right)$ is $\equiv 2\pod{4}$.
\end{theorem}

\begin{proof}
The only primes that can ramify are primes whose restriction to $k$ ramifies in $F$. In view of Lemma \ref{lemma9} the only possibilities are the 2 primes of $k(\sqrt{\epsilon})$ lying over $(\sqrt{p})$; it's easy to see that they both ramify in $F(\sqrt{\epsilon})$. Furthermore $\sqrt{\epsilon}$ is not a local norm at either of these primes. (Because the prime ramifies it suffices to show that the image of $\sqrt{\epsilon}$ in the residue class field is not a square. But Lemma \ref{lemma9} shows that this image is a primitive eighth root of unity in $(\newz/p)^{*}$.  And $p\not\equiv 1\pod{16}$. So in our quadratic extension, $t=2$ and $j$ is even. Since $k(\sqrt{\epsilon})$ has odd class number, Theorem \ref{theorem6} gives the desired result.
\end{proof}


\label{}



\end{document}